\documentclass[12pt]{amsart}
\usepackage[utf8]{inputenc}
\usepackage{amsmath}
\usepackage{amssymb}
\usepackage{amsfonts}
\usepackage{amsthm}
\usepackage{mathrsfs} 
\usepackage{bm}
\usepackage{bbm}
\usepackage{tikz}
\usepackage{centernot}
\usepackage{hyperref}
\usepackage{enumitem}
\usepackage[ruled,vlined]{algorithm2e}
\usepackage[margin=1.2in]{geometry}

\DeclareMathOperator{\im}{Im}

\DeclareMathOperator{\poly}{poly}
\DeclareMathOperator{\tow}{tow}

\newcommand{\E}{\Bbb{E}}
\newcommand{\Z}{\Bbb{Z}}
\newcommand{\R}{\Bbb{R}}

\renewcommand{\P}{\Bbb{P}}

\newcommand{\TT}{\mathcal{T}}

\newcommand{\ep}{\epsilon}

\hypersetup{
    colorlinks=true,
    linkcolor=blue,
    filecolor=magenta,
    urlcolor=blue,
    citecolor=blue,
    backref=true
}

\makeatletter
\newtheorem*{rep@theorem}{\rep@title}
\newcommand{\newreptheorem}[2]{%
\newenvironment{rep#1}[1]{%
 \def\rep@title{#2 \ref{##1}}%
 \begin{rep@theorem}}%
 {\end{rep@theorem}}}
\makeatother

\newreptheorem{prp}{Proposition}

\numberwithin{equation}{section}

\makeatletter
\renewenvironment{proof}[1][\proofname]{\par
  \vspace{-\topsep}% remove the space after the theorem
  \pushQED{\qed}%
  \normalfont
  \topsep0pt \partopsep0pt % no space before
  \trivlist
  \item[\hskip\labelsep
        \itshape
    #1\@addpunct{.}]\ignorespaces
}{%
  \popQED\endtrivlist\@endpefalse
  \addvspace{6pt plus 6pt} % some space after
}
\makeatother

\newtheorem{thm}{Theorem}
\newreptheorem{thm}{Theorem}

\newtheorem{result}{Result}[section]
\newtheorem{lem}[result]{Lemma}
\newtheorem{prp}[result]{Proposition}
\newtheorem{cor}[result]{Corollary}

\newtheorem{clm}[result]{Claim}

\theoremstyle{definition}
\newtheorem{rmk}[result]{Remark}
\newtheorem*{defn}{Definition}

\newtheorem{prob}{Problem}

\theoremstyle{remark}

\newcommand{\hide}[1]{}
\newcommand{\edit}[1]{}%{\color{red}{#1}}}

\newcommand{\rough}[1]{}%\textbf{\textcolor{blue}{#1}}}
\definecolor{darkgreen}{RGB}{75,150,75}
\newcommand{\review}[1]{}%\textcolor{darkgreen}{#1}}

\newcommand{\zh}[1]{\textcolor{blue}{zh: #1}}

\newcommand{\hides}[1]{}%1}
\newcommand{\pub}[1]{}%\textcolor{purple}{#1}}

\title{An improved construction for the triangle removal lemma}
\author{Zach Hunter}
\email{zach.hunter@math.ethz.ch}
\date{\today}

\begin{document}

\maketitle

\begin{abstract}
    We construct $n$-vertex graphs $G$ where $\epsilon n^2$ edges must be deleted to become triangle-free, which contain less than $\epsilon^{(C_{\text{new}}-o(1))\log_2 1/\ep}n^3$ triangles for $C_{\text{new}}= \frac{1}{4\log_2(4/3)} \approx 1.6601$. Previously, a bound of the same shape was known, but with $C_{\text{new}}$ replaced by $C_{\text{old}} := C_{\text{new}}/2$. Our construction uses ideas from additive combinatorics, drawing especially from the corners problem, but does not yield new bounds for those problems.
\end{abstract}

\section{Introduction}

We start by recalling the celebrated \textit{triangle removal lemma} (TRL).\begin{lem}[Triangle Removal Lemma] For every $\ep >0$, there exists $\delta= \delta(\ep)>0$, so that if an $n$-vertex graph $G$ contains $<\delta n^3$ triangles, then there exists a set of $<\ep n^2$ edges which intersects every triangle in $G$.
\end{lem}\noindent The original proof of TRL relied on Szemer\'edi's regularity lemma, which gave a very poor tower-type upper bound for $1/\delta(\ep)$ (i.e., we knew\footnote{Here, we define the tower function `$\tow(x)$' to be a tower of two's with height $\lfloor x\rfloor$, so $\tow(0) := 1$ and $\tow(t):= 2^{\tow(t-1)}$ for $t\ge 1$.} $1/\delta(\ep)\le \tow((1/\ep)^{O(1)})$). Later, a more quantitative upper bound was established by Fox \cite{fox}, who showed $1/\delta(\ep) \le \tow(O(\log(1/\ep)))$. 

Meanwhile, the best-known lower bounds for $1/\delta(\ep)$ are of shape $\exp(C\log^2(1/\ep))$, leaving a large gap. These bounds come from constructions for an even stricter problem.
\begin{defn}
    Given $n\ge 1$, let $\eta(n)$ denote $\max\{\frac{e(G)}{n^2}\}$, over all $n$-vertex graphs $G$ where every edge belongs to exactly $1$ triangle. 
\end{defn}\noindent It is immediate from definitions that $\delta(\ep)<\frac{1}{\max\{n:\eta(n)\ge 3\ep\}}$ (since such $G$ have $e(G)/3<n^2$ total triangles, and we must delete a distinct edge from each such triangle to become triangle-free). Thus for future reference, we note that if we show\footnote{See Subsection~\ref{notation} for our asymptotic notation.} $\eta(n) \gg \exp(-(c+o(1))\sqrt{\ln(n)})$ for some $c>0$, then we get that $1/\delta(\ep) \gg \exp(\frac{1-o(1)}{c^2} \ln^2(1/\ep))$.

To date, the strongest such lower bounds come from additive combinatorics, using variations of a construction due to Behrend \cite{behrend}. For decades, the best known lower bound was
\begin{equation}\label{behrend bound}
    \eta(n) \ge 2^{-(2\sqrt{2}+o(1))\sqrt{\log_2(n)}},
\end{equation}which comes from Behrend's construction of large subsets of $\{1,\dots,N\}$ that lack three term arithmetic progressions. We note that very recent work of the author joint with Elsholtz, Proske, and Sauermann improved upon Behrend's construction (constructing denser progression-free subsets of $\{1,\dots,N\}$), which allows one to replace this constant `$2\sqrt{2}$' by `$2\sqrt{\log_2(24/7)}$' \cite{EHPS}. This would have improved lower bounds for triangle removal, if not for some different progress which happened a few years prior.

Namely, progress was made on the so-called ``corners problem'' (which we define in a later section), allowing for the bound
\begin{equation}\label{green bound}
    \eta(n) \ge 2^{-(2\sqrt{2\log_2(4/3)}+o(1))\sqrt{\log_2 (n)}}.
\end{equation} \hide{The corners problem will be discussed in more detail later on, but it is another well-studied problem in additive combinatorics.} 
\noindent The corners problem is another well-studied problem in additive combinatorics. For many decades, the best bounds for the corners problem came from a direct reduction to the previous problem of avoiding three term arithmetic progressions, which yielded no improvement to Eq.~\ref{behrend bound}. But a few years ago, Green \cite{green} found a more efficient implementation of Behrend-style techniques yielding the state-of-the-art Eq.~\ref{green bound}; this refined a slightly earlier improvement to the corner problem by Linial and Shraibman \cite{linialshraibman}, who were inspired by connections to NOF communication complexity.

\hide{The Corners problem asks to construct large subsets of $\{1,\dots,N\}^2$, lacking triples of the form $(x,y),(x+d,y),(x,y+d)$ with $d$ a non-zero integer. This is another heavily studied problem in additive combinatorics. There is a reduction }

It is perhaps more enlightening to view the constructions from Behrend and Green respectively as each giving a `range of bounds'. Namely, their proofs naturally establish
\begin{equation}\label{behrend range}
    \eta(n) \gg D^{-O(1)}(1/2)^{D}n^{-2/D}
\end{equation}
\begin{equation}\label{green range}\eta(n)\gg D^{-O(1)}(3/4)^{D}n^{-2/D}\end{equation}for any choice of integer $D\ge 1$. By optimizing the parameter $D$, one recovers Eq.~\ref{behrend bound} and Eq.~\ref{green bound}. And it is now obvious why there is an improvement when one replaces the `Behrend range' Eq.~\ref{behrend range} by the `Green range' Eq.~\ref{green range}.

In this short note, we show that $1/\delta(\ep)\gg \exp(C\log^2(1/\ep))$ for an improved constant $C$. Namely, we shall establish a new range of bounds as follows:
\begin{thm}\label{new bounds}
    Given integers $D\ge 1$, $n> (2 D)^{100D}$, we have the bound:
    \[\eta(3n) \gg D^{-O(1)}(3/4)^{D/2}n^{-2/D}.\]
    \hide{Thus we have $\eta(n)\gg D^{-O(1)}(3/4)^{D/2} n^{-2/D}$ for such $n$.} Optimizing $D$,
    \[\eta(n) \gg 2^{-(2\sqrt{\log_2(4/3)}+o(1))\sqrt{\log_2(n)}}.\]
\end{thm}\noindent This important detail here of course, is that $\sqrt{3/4}>3/4$. Phrased in terms of the TRL, we see: 
\begin{thm}\label{main}
    We have that $\delta(\ep) \le \ep^{(C-o(1))\log_2(1/\ep)}$, for $C= \frac{1}{4\log_2(4/3)} \approx 1.6601$.
\end{thm}\noindent The previous best bound had a worse constant of $\frac{1}{2\log_2(4/3)}\approx 0.8301$.

Our argument is inspired by the reduction from the corners problem to TRL. As a corollary of our methods, one also gets following (which Thereom~\ref{new bounds} is implicitly a special case of).

\begin{prp}\label{abstract reduction} 
    Let $X,Y,Z,W$ be additive sets, where $W$ is $3$-AP-free\footnote{Here, we say an additive set is \textit{$3$-AP-free} if it does not contain any set of the form $\{x,x+d,x+2d\}$ with $d\neq 0$ (non-trivial arithmetic progressions of length $3$).}, and $|X|=|Y|=|Z|=n$. Then \[\eta(3n)\ge \frac{1}{9}\P_{x\sim X,y\sim Y}(x+y\in Z\text{ and }x-y\in W).\]
\end{prp}
\begin{rmk}
    Note that here we do not assume any upper bound on the size of $W$, only that it is $3$-AP-free. So this suggests some sort of concentration phenomenon occurs, where if $x+y$ concentrates on a small set, then $x-y$ cannot concentrate on a $3$-AP-free set. 
    
    At first glance, this set-up sounds more abstract, so one might hope to get significantly better bounds for this problem. However, as we will show in the conclusion, if one assumes our additive sets live in a torsion-free group, then the RHS above has a polynomial relationship to the quantity $\frac{r_3(n)}{n}$ (where $r_3(n)$ denotes the maximum cardinality of a $3$-AP-free subset of $\{1,\dots,n\}$). We refer the reader to Proposition~\ref{limit of abstract} and the surrounding discussion for further details.
\end{rmk}
\subsection{Notation}\label{notation}

We use both Big O and Vinagradov notion. Given functions $f=f(n),g=g(n)$, we say $f =O(g),f\ll g, g=\Omega(f)$ to mean there exists an absolute constant $C>0$ so that $f(n)\le Cg(n)$ for all sufficiently large $n$. We write $f=o(g)$ if $f(n)/g(n)\to 0$ as $n\to \infty$.

For positive integer $n$, we write $[n]:=\{1,\dots,n\}$.

\section{Reductions}

The following lemma is a minor variant of the reduction from the corners problem to TRL.
\begin{lem}\label{general construction}
    Let $A$ be a set, $f_1,f_2,f_3$ be maps defined over $A$, so that $|f_i(A)|\le n$ for $i=1,2,3$. Also suppose that whenever $(a_1,a_2,a_3)\in A^3$ satisfies \[f_1(a_2)=f_1(a_3),f_2(a_3)=f_2(a_1),f_3(a_1)=f_3(a_2),\]  we have $a_1=a_2=a_3$. Then $\eta(3n) \ge |A|/9n^2$.
\end{lem}
\begin{proof}
    For $i=1,2,3$, let $V_i := f_i(A)$. By relabelling, we may assume that these sets are disjoint. 

    Now, we define a tripartite 3-uniform hypergraph with vertex set $V:=V_1\cup V_2\cup V_3$ and with hyperedge set $\TT := \{\{f_1(a),f_2(a),f_3(a)\}:a\in A\}$. It is clear that $|V|\le 3n$. We further note that $|\TT| = |A|$. Indeed, our assumptions imply that $\{f_1(a),f_2(a),f_3(a)\} \neq \{f_1(a'),f_2(a'),f_3(a')\}$ for distinct $a,a'\in A$, by considering the tuple $(a,a,a')\in A^3$.

    Now consider the graph $G = (V,E)$, where $uv\in E$ if and only if $uv\subset T$ for some triple $T\in \TT$. Clearly, $G$ contains at least $ |\TT| = |A|$ triangles. Whence the claim will follow if we show that $G$ is ``diamond-free''\footnote{A graph is called diamond-free if no two triangles share an edge. The name is because if one draws the avoided subgraph (\( \tikz[baseline=(v1.base),scale=0.14]{
  \node[draw,circle,inner sep=0.7pt] (v1) at (1,0) {};
  \node[draw,circle,inner sep=0.7pt] (v2) at (2,1) {};
  \node[draw,circle,inner sep=0.7pt] (v3) at (1,2) {};
  \node[draw,circle,inner sep=0.7pt] (v4) at (0,1) {};
  \draw[line width=0.5pt] (v1) -- (v2) -- (v3) -- (v4) -- (v1);
  \draw[line width=0.5pt] (v2) -- (v4);
} \)) it looks (a bit) like a diamond.} (potentially adding dummy vertices if $|V|<3n$).

    We argue as follows. For any triple $v_1\in V_1,v_2\in V_2,v_3\in V_3$ which forms a triangle, we note there must exist $a_{12},a_{23},a_{31}\in A$ so that for $ij\in \binom{[3]}{2}$
    \[f_i(a_{ij})=v_i,f_j(a_{ij})=v_j\](meaning $a_{ij}$ is the reason why $v_iv_j$ is an edge of $G$). Then we clearly have
    \[f_1(a_{31}) = v_1=f_1(a_{12}),f_2(a_{12})=v_2=f_2(a_{23}),f_3(a_{23})=v_3=f_3(a_{31}).\] So we have a contradiction to our assumption unless $a_{12}= a_{23}=a_{31}$.

    This immediately shows that $G$ has exactly $|\TT|$ triangles (i.e., we did not create any additional triangles). It now suffices to check these triangles are edge-disjoint. This can be handled with similar reasoning. Indeed, if (say) $f_1(a) =f_1(a')$ and $f_2(a)=f_2(a')$, then taking $a_{12} = a'$ and $a_{23} = a_{31} =a$, we get $a_{12},a_{23},a_{31}$ (which are not all equal) that correspond to the edges of a triangle in $G$, which is a contradiction (by the reasoning above). 

    Since all the triangles in $G$ are edge-disjoint, we are done.
\end{proof}

Like with the corners problem, we will take $A\subset G^2$ to be an additive set (where $G$ is some group), and $f_1,f_2,f_3$ to be the homomorphisms
\[f_1:(x,y)\mapsto x,\quad f_2:(x,y)\mapsto y,\quad f_3(x,y)\mapsto x+y.\]We say $A$ is \textit{corner-free} if it does not contain any triple $(x,y),(x+d,y),(x,y+d)$ where $d\neq 0$.  

The corners problem asks to bound $r_{\angle}(n)$, the cardinality of the largest corner-free $A\subset [n]^2$. Combining Lemma~\ref{general construction} with Claim~\ref{corners are disjoint} (stated below) shows that $\eta(6n) \gg \frac{r_\angle(n)}{n^2}$ (because $|f_3([n]^2)|\le 2n$).
\begin{clm}\label{corners are disjoint}
    Let $f_1,f_2,f_3$ be defined as above. If $A$ is corner-free, then whenever $(a_1,a_2,a_3)\in A^3$ satisfies \[f_1(a_2)=f_1(a_3),f_2(a_3)=f_2(a_1),f_3(a_1)=f_3(a_2),\] that we have $a_1=a_2=a_3$.
\end{clm}
\begin{proof}
    Supposing $a_3 = (x,y)$, then we must have $a_2 = (x,y'), a_3 = (x',y)$ and $x+y'=x'+y$. Writing $d:= y'-y$, we get that $x'= x+d$. Since $A$ is corner-free, it follows that $d=0$, whence $a_1=a_2=a_3$ as desired.
\end{proof}

In this paper, the ambient group we will work with shall always be $\Z^D$ for an appropriately chosen $D$. Here we will make ample use of following handy observation, which is a variation on ideas dating back to Behrend \cite{behrend}.   

\begin{lem}\label{norm coloring}
    There exists a coloring $C:([M+1]^D)^2\to \{0,\dots,DM^2\}$ so that every color class is corner-free.
\end{lem}
\begin{proof}
    Write $G:= ([M+1]^D)^2$ to denote our ``grid''. %Though $G$ is not a group, it is perhaps useful to think of it as being approximately group-like.\zh{technically we do not use this here, but in applications its useful to have small doubling.} \zh{I've added `perhaps', does this fix things?}

    First, we define the homomorphism \[\psi:(\Z^D)^2\to \Z^D;(\vec{x},\vec{y})\mapsto \vec{x}-\vec{y}.\]We have that $\psi(G) = \{-M,\dots,M\}^D$. We then define the coloring $C_0:\Z^D\to \Z; \vec{x}\mapsto ||\vec{x}||_2^2 = \sum_{i=1}^d x_i^2$. Finally, we shall take $C:= C_0\circ \psi$.

    It remains to confirm $C$ has the desired properties. Firstly, it is clear that $C(G) \subset\{0,\dots,DM^2\}$ (since $\psi(g)\in \Z^D$ has coordinates bounded by $M$ for all $g\in G$).

    Next we must rule out monochromatic corners. Fix some $g\in G$ and non-zero $d\in \Z^D$. We have that
    \[\psi(g+(d,0))-\psi(g) =\psi(g) -\psi(g+(0,d)) =d \neq 0\](meaning the corner $g,g+(d,0),g+(0,d)$ turns into a non-trivial $3$-term arithmetic progression under $\psi$). However, $C_0$ should not contain any non-trivial $3$-term arithmetic progressions.

    Indeed, let $P=\{x_0-d,x_0,x_0+d\}$ be an arithmetic progression in $\Z^D$. Then parallelogram law says that
    \[||x_0-d||_2^2+||x_0+d||_2^2 = 2||x_0||_0^2 +2||d||_2^2.\]So $C_0(x_0-d) =C_0(x_0+d) = C_0(x_0)$ implies that $||d||_2^2 =0$ (whence $P$ is trivial). This establishes $C_0$ does not have monochromatic $3$-term arithmetic progressions, which in turn implies $C$ does not have monochromatic corners, completing the proof.
\end{proof}

\begin{cor}\label{corner reduction}
    Let $X,Y\subset [M+1]^D$ and $Z\subset \Z^D$ be sets of size at most $n$. Then $\eta(3n) \ge \frac{\#(x\in X, y\in Y: x+y\in Z)/n^2}{DM^2+1}$. 

\end{cor}
\begin{proof}
    Let $A_0 \subset X\times Y$ be the set of pairs $(x,y)$ such that $x+y\in Z$. For $a=(x,y)\in A_0$, we set $f_1(a):= x,f_2(a):= y,f_3(a) := x+y$.

    Let $C$ be the coloring of $([M+1]^D)^2\supset X\times Y$ given by Lemma~\ref{norm coloring}. By pigeonhole, we can find $A\subset A_0$ of size $|A|\ge \frac{|A_0|}{|\im(C)|} = \frac{\#(x\in X, y\in Y: x+y\in Z)}{DM^2+1}$, so that $A$ is a corner-free subset. 

    From definitions we have $f_1(A)\subset X, f_2(A)\subset Y, f_3(A)\subset Z$; and by our assumptions we have $|X|=|Y|=|Z|=n$. Thus the result follows from Lemma~\ref{general construction} and Claim~\ref{corners are disjoint}.   \hide{ So, we just need to confirm that $A$ is corner-free. This is simply true for the same reasons as in the proof of Lemma~\ref{norm coloring}. Indeed, suppose there was some $(x,y)\in A$ and non-zero difference $d$ so that $\{(x,y),(x+d,y),(x,y+d)\}\subset A$. Then, from definitions, we get that $P:= \{x-(y+d),x-y,(x+d)-y\}\subset W$, contradicting that $W$ lacks progressions of length $3$. }
\end{proof}

Corollary~\ref{corner reduction} is what shall allow us to deduce Theorem~\ref{main} (yielding improved bounds for $\eta(n)$ and $\delta(\ep)$). We shall carry out this in the next section. But before doing so, let us show how the arguments in Lemma~\ref{norm coloring} and Corollary~\ref{corner reduction} can be slightly abstracted to yield Proposition~\ref{abstract reduction} (which we restate for convenience).

\begin{repprp}{abstract reduction}
    Let $X,Y,Z,W$ be additive sets, where $W$ is $3$-AP-free, and $|X|=|Y|=|Z|=n$. Then \[\eta(3n)\ge \frac{1}{9}\P_{x\sim X,y\sim Y}(x+y\in Z\text{ and }x-y\in W).\]
\end{repprp}
\begin{proof}
     Let $A\subset X\times Y$ be the set of pairs $(x,y)$ where $x+y\in Z$ and $x-y\in W$. For $a=(x,y) \in A$, we define $f_1(a) := x,f_2(a) := y,f_3(a):=x+y$.

     From definitions we have $f_1(A)\subset X,f_2(A)\subset Y,f_3(Z)\subset Z$, and so we get that these all have size at most $n$ (by assumption). Assuming $A$ is corner-free, one could then apply Lemma~\ref{general construction} to get $\eta(3n) \ge \frac{1}{9}\frac{|A|}{n^2}$ giving the desired lower bound (since $\frac{|A|}{n^2} = \P_{x\sim X,y\sim Y}(x+y\in Z\text{ and }x-y\in W)$).

     We are left to confirm that $A$ is indeed corner-free. But this is simple. Indeed, suppose that there was some $d\neq 0$ and $x_0,y_0$ so that $\{(x_0,y_0),(x_0+d,y_0),(x_0,y_0+d)\}\subset A$. Then we get that $\{(x_0-y_0),(x_0-y_0)+d,(x_0-y_0)-d\}\subset W$, contradicting that $W$ was $3$-AP-free.
\end{proof}

\section{The construction} 

With the preparation from the last section, we can easily asymptotically recover the previous bound from Eq.~\ref{green bound}. From the discussion in the introduction, it suffices to show that for any $n,D$, we have $\eta(3n)\gg D^{-1}(3/4)^D n^{-2/D}$. In what follows, we sketch the details, ignoring issues related to rounding for convenience. 

To do this, we simply write $M:= n^{1/D}$, $X=Y=Z= \{-M/2,\dots,M/2\}^D$. We can rewrite $\#(x\in X, y\in Y: x+y\in Z)/n^2$ as $\P_{x\sim X,y\sim Y}(x+y\in Z)$. One observes that \[\P_{x\sim X,y\sim Y}(x+y\in Z) = \P_{a,b\sim \{-M/2,\dots,M/2\}}(a+b\in \{-M/2,\dots,M/2\})^D \ge (3/4)^D.\]Translating $X,Y$ by $(M/2+1,\dots,M/2+1)$ (so they are now sets contained in $[M+1]^D$) and $Z$ by $(M+2,\dots,M+2)$, the probability does not change. Whence Proposition~\ref{corner reduction} gives that $\eta(3n) \ge \frac{(3/4)^D}{DM^2}$ (modulo rounding) as desired.

Our improved bound Theorem~\ref{main} essentially now boils down to showing that the box $[-1/2,1/2]^D$ is not optimally closed under addition. Instead, it will be better to consider the Euclidean ball. Formally, it is known that when $X,Y,Z\subset \R^d$ are subsets of measure one, that $\P_{x\sim X,y\sim Y}(x+y\in Z)$ is maximized when $X=Y=Z$ is the Euclidean ball around the origin (this follows from a result known as the Riesz rearrangement theorem). But we do not need to prove this, as we are merely concerned with lower bounds. 

Instead, we need the following bound. We defer its proof to Subsection~\ref{computations}, as these are not novel calculations.
\begin{prp}\label{additive closure of ball}
    Let $B$ be a Euclidean ball centered around the origin in $\R^D$. Then
    \[\P_{x,y\sim B}(x+y\in B)\gg D^{-O(1)} (3/4)^{D/2}.\]
\end{prp}The main point is that we improve the $(3/4)^D$ probability that a $D$-dimensional box is closed under addition to $(3/4)^{D/2}$ when using a $D$-dimensional ball. By discretizing the ball, and recreating the sketch from before (now being precise about minor rounding issues) we shall be able to derive Theorem~\ref{main}.

To be precise about potential rounding issues, we shall employ a simple averaging argument. Here, we write $\mu(\cdot)$ to denote the Lebesgue measure. 
\begin{lem}\label{shift lower}
    Let $S\subset \R^D$ be a measurable set. Then there is some shift $t\in [0,1)^D$ so that $|(t+S)\cap \Z^D|\ge \mu(S)$.
\end{lem}
\begin{proof}
    Pick $t\sim [0,1)^D$ uniformly at random.     For $\xi\in \Z^D$, $\P( \xi \in t+S) = \mu( S\cap \xi-(0,1]^D)$. By linearity of expectation, we get $\E[|(t+S)\cap \Z^D|] = \sum_{\xi\in \Z^D}\P(\xi \in t+S) = \mu(S)$. Whence by averaging, one can fix some outcome of $t$ where $|(t+S)\cap \Z^D|\ge \mu(S)$, as desired.
\end{proof}
\noindent An identical argument yields:
\begin{lem}\label{shift upper}
    Let $S\subset \R^D$ be a measurable set. Then there is some shift $t\in [0,1)^D$ so that $|(t+S)\cap \Z^D|\le \mu(S)$.
\end{lem}

We shall now derive our main result.

\begin{proof}[Proof of Theorem~\ref{new bounds}] Consider some $D\ge 1$ and $n\ge (2D)^{100D}$. We shall show $\eta(3n)\gg D^{-O(1)} (3/4)^{D/2}n^{-2/D}$. To then derive $\eta(n) \gg 2^{-(2\sqrt{\log_2(4/3)}+o(1))\sqrt{\log_2(n)}}$, one simply takes $D_n= \lceil \frac{2}{\sqrt{\log_2(4/3)}}\sqrt{\log_2(n)}\rceil $, and observes $n/3 \ge (2D_n)^{100D_n}$ for all sufficiently large $n$.

Let $B\subset \R^D$ be a Euclidean ball of radius $r$ around the origin, where $r$ is chosen so that $\mu(B) = n$. Very crudely, we have that $r\le D n^{2/D}$. Otherwise $[-n^{1/D},n^{1/D}]^D \subset B$, contradicting $\mu(B)=n$. From this upper bound, we also see that $B\subset [-Dn^{1/D},Dn^{1/D}]^D$.

Now, we consider $S\subset \R^{2D}$, defined to be the set of $(x_1,\dots,x_D,y_1,\dots,y_D)\in \R^{2D}$ so that $x,y,x+y$ all belong to $B$. Proposition~\ref{additive closure of ball} tells us that $\mu(S)\gg  n^2 D^{-O(1)}(3/4)^{D/2}$. By Lemma~\ref{shift lower}, we can find some $t\in [0,1)^{2D}$ so that $|(t+S)\cap \Z^{2D}|\ge \mu(S)$. 

Set $t^{(1)} := (t_1,\dots,t_D),t^{(2)}:= (t_{D+1},\dots,t_{2D})$. We then define $X_0 := (t^{(1)}+B)\cap \Z^D,Y_0 := (t^{(2)} +B)\cap \Z^D, Z_0 := B\cap \Z^D$. We claim $|X_0|,|Y_0|,|Z_0|\le 2n$. Assuming this claim, standard averaging finds $X_1\subset X_0,Y_1\subset Y_0,Z_1 \subset Z_0$ all of size at most $n$, so that \[\#(x\in X_1,y\in Y_1: x+y\in Z_1)\ge (1/8)\#(x\in X_0,y\in Y_0:z\in Z_0) \gg D^{-O(1)} (3/4)^{D/2} n^2.\] And so, noting that $X_1,Y_1,Z_1\subset [-Dn^{1/D}-1,Dn^{1/D}+1]^D\cap \Z^D$, we can translate them to live inside $[M+1]^D$ for $M\ll Dn^{1/D}$. Whence, Lemma~\ref{corner reduction} gives
\[\eta(3n)\gg \frac{D^{-O(1)}(3/4)^{D/2}}{D M^2}\gg \frac{D^{-O(1)}(3/4)^{D/2}}{D^3 n^{2/D}}= D^{-O(1)}(3/4)^{D/2}n^{-2/D}.\]

To finish the proof, we must justify why $X_0,Y_0,Z_0$ each have cardinality at most $2n$. We shall only discuss $X_0$, as the remaining cases are identical.

We have that $X_0 = (t^{(1)}+B)\cap \Z^D$. Set $r' := r+\sqrt{D}$. Let $B'$ be the Euclidean ball of radius $r'$, and note that for any $s\in [0,1)^D$, we have $s+B'\supset B$ (by the triangle inquality, if $v\in B$, then $||v-s||_2\le ||v||_2+||s||_2\le r+D^{1/2}=r'$, meaning $v\in s+B'$).

Now, since we assumed $n>(2D)^{100D}$, we must have that $r\ge D^{50}$, otherwise $B\subset (-D^{100},D^{100})^D$ which has volume less than $n$. Thus, $r'\le (1+D^{-2})r$, meaning $\mu(B')\le (1+D^{-2})^D\mu(B)\le 2n$. Now, we just apply Lemma~\ref{shift upper} to find some $s\in [0,1)^D$ with $|s+t^{(1)}+B'\cap \Z^D|\le \mu(B')\le 2n$. Recalling $X_0 = (t^{(1)}+B)\cap \Z^D$, the inclusion $t^{(1)}+B\subset t^{(1)}+s+B'$ gives $|X_0|\le 2n$ (as required).
\end{proof}

\subsection{The computation}\label{computations}

Throughout this subsection, we let $B$ denote the unit ball in $\R^D$.

We recall the following fact (see e.g., \cite{MO}).
\begin{clm}\label{dot product pdf}
    Let $u,v\sim S^{D-1}$ be random unit vectors in $\R^D$, and consider the random variable $R:= \langle u,v\rangle$. Writing $f_R$ for the pdf of this random variable, we have\footnote{Here $\Gamma(x):= \int_0^\infty t^{x-1}e^{-t}\,dt$ is the gamma function.}
    \[f_R(r) = \frac{\Gamma(D/2)}{\sqrt{\pi}\Gamma((D-1)/2)}(1-r^2)^{\frac{D-3}{2}}\mathbf{1}(|r|<1).\]
\end{clm} \noindent At a high level, we shall argue:
\begin{align}
    \P_{x,y\sim B}(x+y\in B) &\ge \P_{u,v\sim S^{D-1}}(||u+v||_2\le 1) \label{step1}\\ 
    &= \P_{u,v\sim S^{D-1}}(\langle u,v\rangle \le -1/2)\label{step2}\\
    &\gg \int_{-1}^{-1/2} (1-r^2)^{\frac{D}{2}}\,dr \label{step3}\\
    &\gg (1/D)(3/4)^{D/2}.\label{step4}
\end{align}
\noindent It is clear that Eq.~\ref{step2} holds. So it just remains to establish the three other lines.

Let us now justify Eq.~\ref{step1}.
\begin{lem}
    We have
    \[\P_{x,y\sim B}(x+y\in B)\ge \P_{u,v\sim S^{D-1}}(u+v\in B).\]
\end{lem}
\begin{proof}
    Let $x,y$ be (independent) uniformly random elements of $B$. Write $\overline{x} := \frac{x}{||x||_2},\overline{y} := \frac{y}{||y||_2}$, and note that these are uniformly random elements of $S^{D-1}$ (which are also independent). It suffices to show that $\overline{x}+\overline{y}\in B$ implies $x+y\in B$. 

    So, suppose $\langle \overline{x},\overline{y}\rangle \le-1/2$. Then $||x+y||_2^2 = ||x||_2^2-2\langle x,y\rangle +||y||_2^2$. Supposing WLOG $\lambda_x:= ||x||_2\ge ||y||_2:= \lambda_y$, this becomes $\lambda_x^2-2\lambda_x\lambda_y\langle\overline{x},\overline{y}\rangle +\lambda_y^2\le \lambda_x^2\le 1$. So $x+y\in B$ (as desired).
\end{proof}

Now, we give an approximation which (together with Claim~\ref{dot product pdf}) immediately implies Eq.~\ref{step3}.
\begin{lem}
    Let $f_R(r)$ be defined as it was in Claim~\ref{dot product pdf}, for some $D\ge 2$. Then 
    \[f_R(r)\gg (1-r^2)^{D/2}\mathbf{1}(|r|<1).\]
\end{lem}
\begin{proof}
    Note $\Gamma(x)$ is increasing for $x\ge 2$, and $(1-r^2)^{-3/2}\ge 1$ when $|r|<1$. Thus $f_R(r)\ge \frac{1}{\sqrt{\pi}}(1-r^2)^{D/2}\mathbf{1}(|r|<1)$ for $D\ge 5$. The result follows (after changing the implicit constant to handle $D=2,3,4$). 
\end{proof}
Finally, we compute
\[\int_{-1}^{-1/2}(1-r^2)^{D/2}\,dr \ge (1/D)(1- (1/D-1/2)^2)^{D/2} \ge (1/D) (3/4-1/D)^{D/2}\gg (1/D)(3/4)^{D/2},\]confirming Eq.~\ref{step4}. Stringing these together, we get 
\[\P_{x,y\sim B}(x+y\in B)\gg (1/D)(3/4)^{D/2},\]proving Proposition~\ref{additive closure of ball}.

\section{Conclusion}
As promised in the introduction, we will now establish a limitation of Proposition~\ref{abstract reduction}. Afterwards, we will speculate a bit about future directions.

Recall that $r_3(n)$ denotes the maximum cardinality of a $3$-AP-free subset of $\{1,\dots,n\}$. Let us define $p_3(n)$ to be the maximum of $\P_{x\sim X,y\sim Y}(x+y\in Z\text{ and }x-y\in W)$, over all additive sets $W,X,Y,Z$ where $W$ is $3$-AP-free and $X,Y,Z$ all have cardinality $n$. Proposition~\ref{abstract reduction} showed that $\eta(3n)\ge (1/9)p_3(n)$. 

Now, a simple averaging argument reveals that $p_3(n)\gg \frac{r_3(n)}{n}$. Indeed, fix some $3$-AP-free set $A\subset \{1,\dots,n\}$ of size $r_3(n)$. Take $X=Y=Z= \{1,\dots,n\}$. There are at least $(n/2)^2$ choices of $x\in X,y\in Y$ so that $x+y\in Z$. And for each such $x,y$ and $a\in A$, there exists some choice of $t\in \{-2n,\dots,n-1,n\}$, so that $x-y=a+t$. So by pigeonhole, there exists some choice of $t$ so that we have at least $\frac{(n/2)^2|A|}{3n+1}= \Omega(r_3(n) n)$ choices of $x\in X,y\in Y$ with $x+y\in Z$ and $x-y\in t+A$. Dividing this count by $n^2$ gives a lower bound on $p_3(n)$ (explicitly we get $p_3(n) \ge \frac{r_3(n)}{4(3n+1)}$).

We now observe there is also a bound in the other direction.
\begin{prp}\label{limit of abstract}
    We have that $p_3(n)\ll \left(\frac{r_3(n)}{n}\right)^{\Omega(1)}$.
\end{prp}
To this end, we recall an old theorem of Ruzsa.
\begin{thm}\label{ruzsa doubling}
    Let $A\subset \Z$ be a $3$-AP-free set of size $n$. Then $|A+A|\ge (1/2)(n/r_3(n))^{1/4} n$.
\end{thm}
 We also need a somewhat precise form of the Balogh-Szemer\'edi-Gowers lemma, combined with the Ruzsa triangle inequality. We note that such formulations originally arose from work of Bourgain on the Kakeya conjecture \cite{bourgain}.
 \begin{lem}\label{kakeya BSG}
 Let $A,B$ be additive sets each with cardinality $n$, and let $E\subset A\times B$. Suppose $|E|\ge cn^2$ and 
 \[|\{a+b: (a,b)\in E\}|\le n.\]
 Then we can find $A'\subset A,B'\subset B$ so that $|A'+A'- B'-B'|\ll (1/c)^{O(1)} n$ and $|E\cap (A'\times B')|\gg c^{O(1)}n^2$.
 \end{lem} \noindent The above can be derived from exercises in Tao-Vu. Specifically \cite[Exercise 6.4.10]{taovu} finds sets $A',B'$ where $|E\cap (A'\times B')|$ is sufficiently large and $|A'-B'|\ll (1/c)^{O(1)}|A'|^{1/2}|B'|^{1/2}$. From here, the second condition can be strengthened to get the desired upper bound on $|A'+A'-B'-B'|$ by using \cite[Proposition 2.27(v)]{taovu} (a form of the Plunnecke-Ruzsa inequality).

\begin{proof}[Proof of Proposition~\ref{limit of abstract}]
    Suppose $p_3(n) = p$. Fix additive sets $W,X,Y,Z$ demonstrating this bound. WLOG, we may assume our sets live in $\Z$, as otherwise we can take a generic homomophism (since we are torsion-free).

    Let $E\subset X\times Y$ be the set of pairs $(x,y)$ so that $x+y\in Z$ and $x-y\in W$. We have $|E|= pn^2$. Also,
    \[\{x+y:(x,y)\in E\}\subset Z.\]Recalling $|X|=|Y|=|Z| =n$, we can apply Lemma~\ref{kakeya BSG} with $c:= p$ (and $A:=X,B:= Y$) to find $X'\subset X,Y'\subset Y$ obeying
    \[|E\cap (B'\times Y')|\ge p' n^2\quad \text{and}\quad |X'+X'-Y'-Y'|\le Kn\]where $p'\gg p^{O(1)}$ and $K\ll (1/p)^{O(1)}$. Let $W':= W\cap (X'-Y')$. By pigeonhole, we can find some $x'\in X'$ with $p'n$ choices of $y'\in Y'$ satisfying $(x',y')\in E$; as $y'-y'_1\neq x'-y'_2$ for distinct $y_1',y_2'$, we get $|W'|\ge p'n$.  

    For convenience, write $n':= p'n$. WLOG, let us assume $|W'|= n'$ (deleting a few extra elements if necessary). Since $W'\subset X'-Y'$, we get that $|W'+W'|\le |X'+X'-Y'-Y'| \le Kn= (K/p')|W'|$. Since $W'\subset W\subset \Z$ and $W$ is $3$-AP-free, Theorem~\ref{ruzsa doubling} gives $\frac{|W'+W'|}{|W'|}\ge (1/2) (n'/r_3(n'))^{1/4}$. Rearranging things, we get that
    \[\frac{n'}{r_3(n')}\le 16 (K/p')^4 \ll (1/p)^{O(1)}.\] To finish, it will suffice to show that $\frac{n}{r_3(n)}\ll (\frac{n'}{r_3(n')})^{O(1)}$. 
    
    But this is mundane. Recalling $p_3(n)\gg r_3(n)/n$, we get that $n'\gg (r_3(n)/n)^{O(1)} n$. Combined with the fact that $r_3(n)/n\ge n^{-o(1)}$ (cf. \cite{behrend}), we get that $n'\ge \sqrt{n}$ for all large $n$. We finally claim that $\frac{r_3(n)}{n}\gg \left(\frac{r_3(n')}{n'}\right)^2$. First note that $r_3(100n'^2)\ge r_3(n')^2$ since if $A\subset [n']$ is $3$-AP-free, so is $\{a_1+10n'a_2:a_1,a_2\in A\}$. And then since $100n'^2\ge n$, one can show that 
    \[\frac{r_3(n)}{n}\gg \frac{r_3(100n'^2)}{100n'^2}\](fix any set $A\subset [100n'^2]$ and simply pick $t\in \{-100n'^2,\dots,100n'^2\}$ randomly; from expectations there will exist some outcome $t$ where $|(t+A)\cap [n]|\ge \frac{|A|}{200n'^2+1}n$). This establishes $\frac{n}{r_3(n)} \ll (\frac{n'}{r_3(n')})^{O(1)}$, which in turn establishes $\frac{n}{r_3(n)}\ll (1/p)^{O(1)}$, completing our proof.
\end{proof} 

\subsection{Final thoughts}

In light of Proposition~\ref{limit of abstract}, it seems almost certain that Proposition~\ref{abstract reduction} cannot dramatically improve the bounds for TRL (much to the dismay of any optimists\footnote{The author does not hold this view, though we know a person or two who does.} who believe (say) $\delta(\ep)\le \exp(-\poly(1/\ep))$). Indeed, the celebrated breakthrough work of Kelley and Meka \cite{KM} established that $\frac{r_3(n)}{n}\le \exp(-\Omega(\log^{1/12}(n)))$, so the same shape of bound will hold for $p_3(n)$ (only changing the constant hidden by $\Omega$). 

Additionally in very recent work, Jaber, Liu, Lovett, Osthuni, and Sawhney established quasipolynomial bounds for the corners problem. Recalling $r_{\angle}(n)$ is the maximum cardinality of a corner-free subset of $[n]^2$, the main result of their paper is that $r_\angle(n) \ll \exp(-\Omega(\log^{1/1200}(n)))n^2$ (cf. \cite[Corollary~1.7]{JLLOS})! While the work of our paper did not improve the best bounds for $r_\angle(n)$ per se, one can easily embed our $D$-dimensional construction to get a corner-free set inside $[2^Dn]^2$. And in fact, a generalization of the arguments in Proposition~\ref{limit of abstract} shows that if $A\subset \Z^2$ is a corner-free set with $\max\{|\{x:(x,y)\in A\}|,|\{y:(x,y)\in A\}|,|\{x+y:(x,y)\in A\}|\}\le n$, then $|A|/n^2\ll \left(\frac{r_\angle(n)}{n^2}\right)^{\Omega(1)}$.

Of course, there is still potential for someone to show (say) $\eta(n)\le \exp(-o(\sqrt{\log n}))$ via the corners problem, which the author would find incredibly interesting. And for those who hope to show that $\eta(n)$ does not have quasipolynomial bounds, there is still the `multicolored' variants of additive problems like $r_3(n)$ or $r_{\angle}(n)$. The interested reader is referred to \cite{CFTZ} and \cite{KSS}, which get nearly tight bounds for multicolored problems in the finite field setting.


\begin{thebibliography}{}
\bibitem{behrend} F. A. Behrend, \textit{On sets of integers which contain no three in arithmetic progression,} in
\textit{Proceedings of the National Academy of Sciences} \textbf{32} (1946), p. 331-332.

\bibitem{bourgain} J. Bourgain, \textit{On the Dimension of Kakeya Sets and Related Maximal Inequalities,} in \textit{Geometric and Functional Analysis} \textbf{9} (1999), p. 256-282.

\bibitem{CFTZ} M. Christandl, O. Fawzi, H. Ta, and J. Zuiddam, \textit{Larger Corner-Free Sets from Combinatorial Degenerations,} in 13th Innovations in Theoretical Computer Science Conference (ITCS 2022). Leibniz International Proceedings in Informatics (LIPIcs), Volume 215.


\bibitem{elkin} M. Elkin, \textit{An improved construction of progression-free sets,} in \textit{SODA 10’} (2010), p. 886-905.

\bibitem{EHPS} C. Elsholtz, Z. Hunter, L. Proske, and L. Sauermann, \textit{Improving Behrend's construction: Sets without arithmetic progressions in integers and over finite fields,} preprint (June 2024), 15 pp. \url{https://arxiv.org/abs/2406.12290}

\bibitem{fischer et al}  E. Fischer, E. Lehman, I. Newman, S. Raskhodnikova, R. Rubinfeld, and
A. Samorodnitsky, \textit{Monotonicity testing over general poset domains,} in \textit{Proceedings of
the thirty-fourth annual ACM symposium on Theory of computing} (2002), p. 474–483.

\bibitem{fox} J. Fox, \textit{A new proof of the graph removal lemma,} in \textit{Annals of Mathematics} \textbf{174} (2011), p. 561–579.

\bibitem{FHS} J. Fox, H. Huang, and B. Sudakov, \textit{On graphs decomposable into induced matchings of linear sizes,} in \textit{Bulletin of the London Mathematical Soceity} \textbf{49} (2016), p. 45-57.

\bibitem{green} B. J. Green, \textit{Lower bounds for corner-free sets,} in \textit{New Zealand Journal of Mathematics} \textbf{51} (2021), p. 1-2.
\bibitem{greenwolf} B. J. Green and J. Wolf, \textit{A note on Elkin’s improvement of Behrend’s construction,} in \textit{Additive Number Theory}, p. 141–144, Springer, New York 2010.

\bibitem{JLLOS} M. Jaber, Y. Liu, S. Lovett, A. Osthuni, and M. Sawhney, \textit{Quasipolynomial bounds for the corners theorem,} preprint (April 2025), 85 pp. \url{https://arxiv.org/abs/2504.07006} 

\bibitem{KM} Z. Kelley and R. Meka, \textit{Strong bounds for 3-Progressions,} in \textit{IEEE 64th Annual Symposium on Foundations of Computer Science (FOCS)} (2023), p. 933-973.

\bibitem{KSS} R. Kleinberg, D. Speyer, and W. Sawin, \textit{The growth rate of tri-colored sum-free sets,} in \textit{Discrete Analysis} \textbf{12} (2018), 10 pp.

\bibitem{linialshraibman} N. Linial and A. Shraibman, \textit{Larger Corner-Free Sets from Better NOF Exactly-$N$ Protocols,} in \textit{Discrete Analysis} \textbf{19} (2021), 9 pp.

\bibitem{MO} I. Pinelis, \textit{Answer to ``Scalar product of random unit vectors'',} \url{https://mathoverflow.net/a/361614/130484}, 2020.

\bibitem{pratt} K. Pratt, \textit{A note on Ordered Ruzsa-Szemer\'edi graphs,} preprint (February 2025), 4 pp. \url{https://arxiv.org/abs/2502.02455}

\bibitem{ruzsa} I. Ruzsa, \textit{Arithmetical progressions and the number of sums,} in \textit{Periodica Mathematica Hungarica} \textbf{25} (1992), p. 105-111.

\bibitem{taovu} T. Tao and V. Vu, \textit{Additive Combinatorics,} Cambridge Stud. Adv. Math.,
105 Cambridge University Press, Cambridge, 2006, xviii+512 pp.

\end{thebibliography}
\end{document}